\documentclass[epsfig,amsfonts]{amsart}
\usepackage{color}

\newcommand{\NN}{\ensuremath{\mathbb{N}}}

\newcommand{\PP}{\ensuremath{\mathbb{P}}}
\newcommand{\QQ}{\ensuremath{\mathbb{Q}}}
\newcommand{\RR}{\ensuremath{\mathbb{R}}}

\newcommand{\TT}{\ensuremath{\mathbb{T}}}

\newcommand{\ZZ}{\ensuremath{\mathbb{Z}}}

\renewcommand{\phi}{\varphi}

\newcommand{\eps}{\ensuremath{\epsilon}}

\newcommand{\bB}{\ensuremath{\mathcal{B}}}

\newcommand{\fF}{\ensuremath{\mathcal{F}}}

\newcommand{\tT}{\ensuremath{\mathcal{T}}}

\newtheorem{theorem}{Theorem}

\newtheorem{lemma}[theorem]{Lemma}

\newtheorem{remark}[theorem]{Remark}


   \thanks{Y. Chen and H. Gao: Supported by a China NSF Grant
11171158, National Basic Research Program of China (973 Program) No. 2013CB834100,  the Natural Science Foundation of Jiangsu Province (BK2011777),
Qing Lan and "333" Project of Jiangsu Province and the NSF of the Jiangsu Higher Education
Committee of China (11KJA110001).\\
M.J. Garrido-Atienza and B. Schmalfu{\ss}: Partially supported by the {\em European Funds for Regional Development} and Ministerio de Economia y Competitividad (Spain) under grant MTM2011-22411.}

\begin{document}

\title[Pathwise solutions of SPDEs and random dynamical systems]
{Pathwise solutions of SPDEs driven by H{\"o}lder-continuous integrators with exponent larger than $1/2$ and random dynamical systems}

\author{Y. Chen}
\address[Yong Chen]{Institute of Mathematics\\
School of Mathematical Sciences, Nanjing Normal University\\Nanjing 210046, China\\ }
\email[Yong Chen]{youngchen329@126.com}

\author{H. Gao}
 \address[Hongjun Gao]{Institute of  Mathematics\\
School of Mathematical Sciences, Nanjing Normal University\\Nanjing 210046, China\\ }
\email[Hongjun Gao]{gaohj@njnu.edu.cn}

\author{M. J. Garrido-Atienza}\address[Mar\'{\i}a J. Garrido-Atienza]{Dpto. Ecuaciones Diferenciales y An\'alisis Num\'erico\\Universidad de Sevilla, Apdo. de Correos 1160, 41080-Sevilla,Spain} \email[Mar\'{\i}a J. Garrido-Atienza]{mgarrido@us.es}

\author{B. Schmalfu{\ss }}
\address[Bj{\"o}rn Schmalfu{\ss }]{Institut f\"{u}r Mathematik\\
Institut f{\"u}r Stochastik, Ernst Abbe Platz 2, 07737\\Jena,Germany\\ }\email[Bj{\"o}rn Schmalfu{\ss }]{bjoern.schmalfuss@uni-jena.de}

\begin{abstract}
This article is devoted to the existence and uniqueness of pathwise solutions to stochastic evolution equations, driven by a H\"older continuous function with H\"older exponent in $(1/2,1)$, and with nontrivial multiplicative noise. As a particular situation, we shall consider the case where the equation is driven by a fractional Brownian motion $B^H$ with Hurst parameter $H>1/2$. In contrast to the article by Maslowski and Nualart \cite{MasNua03}, we present here an existence and uniqueness result in the space of H{\"o}lder continuous functions with values in a Hilbert space $V$. If the initial condition is in the latter space this forces us to consider solutions in a different space, which is a generalization of the H{\"o}lder continuous functions. That space of functions is appropriate to introduce a non-autonomous dynamical system generated by the corresponding solution to the equation. In fact, when choosing $B^H$ as the driving process, we shall prove that the dynamical system will turn out to be a random dynamical system, defined over the ergodic metric dynamical system generated by the infinite dimensional fractional Brownian motion.
\end{abstract}

\maketitle

\section{Introduction}
During the last years, there has been an increasing interest in solving stochastic partial differential equations (SPDEs) beyond the Brownian motion case. One of these attempts is given by the {\it Rough Path Theory}, and we refer to the monographs by Lyons and Qian \cite{Lyons} and Friz and Victoir \cite{FV10} for a comprehensive presentation of this theory. When the driving process is a H\"older continuous function with H\"older exponent greater than $1/2$, a different technique to solve these equations consists of using fractional calculus tools, see for example the papers \cite{GLS09, GrAnh, GuLeTin, MasNua03, NuaRas02, TinTuVi}, to name only a few of them.

We generalize the fractional calculus methods in order to give a pathwise meaning to the solutions of some general nonlinear infinite-dimensional evolution equations, associated to an analytic semigroup and driven by a H\"older continuous process with H\"older exponent greater than $1/2$. The existence of solutions of this type has been already studied for instance in Maslowski and Nualart \cite{MasNua03} and in Garrido-Atienza {\it et al.} \cite{GLS09}. In both papers the driving process $B^H$ is a regular fractional Brownian motion (fBm), in the sense that the Hurst index $H\in (1/2,1)$, and in both of them, the stochastic integral with integrator $B^H$ is defined following the theory developed by Z\"ahle \cite{Zah98}, which is based on the so-called fractional derivatives (see Section 2). For these integrals the usual adaptedness assumptions are not necessary. The phase spaces taken in those articles are somehow not as natural as the space of H\"older continuous functions with adequate H\"older exponents, the space that we shall consider in this article as natural phase space. However, at that point we need to afford the problem that the semigroup generated by the linear part of the equation is not H{\"o}lder continuous on any interval $[0,T]$. Nevertheless, in spite of this lack of regularity, we are able to obtain the existence of solutions in a modification of the space of H\"older continuous functions, and to the best of our knowledge this is the first time that such a modification of the H{\"o}lder continuous space is considered to overcome such a difficulty. However, this modification forces us to focus very precisely on checking that the fractional derivatives for the corresponding kind of functions are well-defined. After taking into account all these facts, we are able to prove the existence of a unique solution, by using the Banach fixed point theorem applied pathwise with respect to an adequate equivalent norm of the function spaces mentioned above. \\
This paper can be seen in fact as the base for other projects whose main aim is to study asymptotical properties of SPDEs driven by any fBm, by exploiting the pathwise sense of their corresponding solutions. On the one hand, the results presented in this article cover the cases $H>1/2$ but do not the white noise case, i.e., when $H=1/2$. For the existence of pathwise solutions for this case (and more general for the cases in which $H\in (1/3,1/2]$) we refer to Garrido-Atienza {\it et al.} \cite{GLS12a}, \cite{GLS12b} and the forthcoming paper \cite{Gar12}. We want to emphasize here that the techniques to obtain such a pathwise solution are much more involved and qualitatively different from the methods that we shall present in this article. On the other hand, we wish to analyze the asymptotic behavior of these pathwise solutions by means of the random dynamical systems theory, and therefore one may wonder whether these solutions could generate what is known as a {\it cocycle}. This question is course difficult to answer at a first glance, because it is well-known that a large class of partial differential equations with stationary random coefficients and It\^o stochastic ordinary differential equations generate random dynamical systems (see Arnold \cite{Arn98}), but for the stochastic partial differential equations driven by the standard Brownian motion the problem is rather unsolved. The main obstacle is that the stochastic It\^ o integral is only defined almost surely where the exceptional sets may depend on the initial state. But as far as the cocycle property is concerned this fact contradicts the definition of a random dynamical system, since initial state dependent exceptional sets are not permitted. As pointed out in \cite{GLS09}, the main advantage of the pathwise integration with respect to the classical It\^o integration theory in this context is that we can avoid this dependence on exceptional sets, and this means that we will be able to study the random dynamical system associated to the corresponding evolution equation. This article is therefore the starting point to some other projects in which we aim at analyzing asymptotical properties of solutions to stochastic evolutions equations, as for instance, investigating the existence and structure of the random attractor associated to those equations. Actually, we have already given some previous considerations in that direction when considering stochastic ordinary equations driven by a fBm with $H>1/2$, see \cite{GMS08}.\\
This article is organized as follows. In Section 2 we introduce the definition and important properties of the integral having a H\"older continuous function as integrator. Section 3 is devoted to show that evolution equations driven by such a functions have a pathwise mild solution, with the property of generating a non-autonomous dynamical system. In the last section of the article we consider the fractional Brownian motion case. Having already obtained a non-autonomous dynamical system in the previous section, we shall focus on the measurability properties needed to claim that this dynamical system is also a random dynamical system, provided that the ergodic metric dynamical system defined by the fBm is considered.\\
\section{Preliminaries}

\subsection{Dynamical systems}\label{s2}
Let $(V,|\cdot|)$ be a Banach space and let $\TT^+=\RR^+$ or $\ZZ^+$. A mapping $\phi:\TT^+\times V\to V$
having the {\em semigroup} property

\begin{equation*}
    \phi(t,\cdot)\circ\phi(\tau,u_0)=\phi(t+\tau,u_0),\qquad\phi(0,u_0)=u_0\qquad\text{for }t,\,\tau\in\TT^+\quad\text{and }u_0\in V
\end{equation*}

is called an {\em autonomous dynamical system}. \medskip

We want to consider a generalization of the concept of an autonomous dynamical system to {\em non-autonomous} and {\em random dynamical systems}. As a first ingredient, for the time set $\TT=\RR$ or $\ZZ$, we introduce the {\em flow} $(\theta_t)_{t\in\TT}$ on the set $\Omega$ of {\em non-autonomous perturbations} by
\begin{align*}
&\theta:\TT\times \Omega\to\Omega\\
    &\theta_t\circ\theta_\tau=\theta_{t+\tau},\quad\theta_0\omega=\omega\quad\text{for }t,\tau\in \TT,\;\omega\in\Omega.
\end{align*}


As a generalization of the semigroup property we consider a {\em cocycle}, which is given by a mapping
\begin{equation*}
    \phi:\TT^+\times\Omega\times V\to V
\end{equation*}
such that
\begin{equation}\label{cocycle}
\begin{split}
 \phi(0,\omega,u_0)&=u_0
 ,\\
   \phi(t+\tau,\omega,u_0)&=\phi(t,\theta_\tau\omega,\cdot)\circ\phi(\tau,\omega,u_0),   \,\text{for all }t,\,\tau\in\TT^+,\;u_0\in V,\;\omega\in\Omega.
    \end{split}
\end{equation}

$\phi$ is also called a {\em non-autonomous dynamical system.}\\

Let us now equip $(\Omega,\theta)$ with a measurable structure. In order to do that, we consider the probability space $(\Omega,\fF,\PP)$, where $\fF$ is a $\sigma$-algebra on $\Omega$ and $\PP$ is a measure, assumed to be invariant and {\em ergodic} with respect to $\theta$.
Then $(\Omega,\fF,\PP,\theta)$ is called a {\em metric dynamical system}.

A $\bB(\TT^+)\otimes\fF\otimes \bB(V),\bB(V)$
measurable mapping $\phi$ having the cocycle property (\ref{cocycle}) is called a {\em random dynamical system (RDS)} with respect to the
metric dynamical $(\Omega,\fF,\PP,\theta)$.

\subsection{Integrals in Hilbert spaces for H{\"o}lder continuous integrators with H{\"o}lder exponent $>1/2$}\label{ss2.2}

Assume ($V$, $|\cdot|$) is a separable Hilbert space. In addition, let $-A$ be a strictly positive and symmetric operator
with a compact inverse which  is the {\em generator} of an {\em analytic exponential decreasing  semigroup} $S$ on $V$. We also introduce the spaces $V_\delta:=D((-A)^\delta)$ with norm $|\cdot|_{V_\delta}$ for $\delta\ge 0$ such that $V=V_0$. The spaces $V_\delta,\,\delta>0$ are continuously embedded in $V$. Let $(e_i)_{i\in\NN}$ be the complete orthonormal base in $V$ generated by the eigenelements of $-A$ with associated eigenvalues $(\lambda_i)_{i\in\NN}$.\\
Let $L(V_\delta,V_\gamma)$ denote the space of continuous linear operators from $V_\delta$ into $V_\gamma$.
Then there exists a constant $c>0$ such that we have the estimates
\begin{align}
  |S(t)|_{L(V, V_{\gamma})}&= |(-A)^\gamma S(t)|_{L(V)}\le
  \frac{c_{S,\gamma}}{t^\gamma}e^{-\lambda t}\qquad\text{for }
  \gamma>
  0\label{eq1},
  \end{align}
  \begin{align}
 |S(t)-{\rm id}|_{L(V_{\sigma+\mu},V_{\theta+\mu})} &\le c
t^{\sigma-\theta}, \quad \text{for }\theta\geq 0,\quad \sigma\in
[\theta,1+\theta],\quad \mu\in\RR \label{eq2}.
\end{align}
In (\ref{eq1}) notice that $\lambda$ is a positive constant. We also note that from these inequalities, for $0\leq q\leq r\leq s\leq t$, we can derive that
\begin{align}\label{eq30}
\begin{split}
& |S(t-r)-S(t-q)|_{L(V_{\delta},V_{\gamma})}\le c(r-q)^\alpha(t-r)^{-\alpha-\gamma+\delta},\\
 & |S(t-r)- S(s-r)-S(t-q)+S(s-q)|_{L(V)}\\
\leq &
c(t-s)^{\beta}(r-q)^{\gamma}(s-r)^{-(\beta+\gamma)}.
\end{split}
\end{align}

\bigskip

Next we introduce some proper spaces where later on we shall investigate the existence of pathwise solutions to stochastic evolution systems.
Let $C^\beta([T_1,T_2];V)$ be the Banach space of {\em H{\"o}lder continuous} functions with exponent $\beta>0$ having values in $V$. A norm on this space is given by

\begin{equation*}
    \|u\|_{\beta}=\|u\|_{\beta,T_1,T_2}=\sup_{s\in [T_1,T_2]}|u(s)|+|||u|||_{\beta,T_1,T_2},
    \end{equation*}
with
$$|||u|||_{\beta,T_1,T_2}=\sup_{T_1\le s<t\le T_2}\frac{|u(t)-u(s)|}{|t-s|^\beta}.$$

$C([T_1,T_2];V)$ denotes the space of  continuous functions on $[T_1,T_2]$ with values in $V$ with finite supremum norm, and let $C^{\beta,\sim}([T_1,T_2];V) \subset C([T_1,T_2];V)$ equipped with the norm

\begin{equation*}
    \|u\|_{\beta,\sim}=\|u\|_{\beta,\sim,T_1,T_2}=\sup_{s\in[T_1,T_2]}|u(s)|+\sup_{T_1< s<t\le T_2}(s-T_1)^\beta\frac{|u(t)-u(s)|}{|t-s|^\beta}.
\end{equation*}

For every $\rho>0$ we can consider the equivalent norm

\begin{align*}
    \|u\|_{\beta,\rho,\sim}=\|u\|_{\beta,\rho,\sim,T_1,T_2}&=\sup_{s\in[T_1,T_2]}e^{-\rho(s-T_1)}|u(s)|\\
    &+\sup_{T_1< s<t\le T_2}(s-T_1)^\beta e^{-\rho(t-T_1)}\frac{|u(t)-u(s)|}{|t-s|^\beta}.
\end{align*}

\begin{lemma}\label{l1}
$C^{\beta,\sim}([T_1,T_2];V)$  is a Banach space.
\end{lemma}

\begin{proof}
Let $(u_n)_{n\in \NN}$ be a Cauchy sequence in $C^{\beta,\sim}([T_1,T_2];V)$. Then it is
straightforward that $(u_n)_{n\in\NN}$ tends to $u_0$ in $C([T_1,T_2];V)$. Let us prove that $(u_n)_{n\in\NN}$ tends also to $u_0\in
C^{\beta,\sim}([T_1,T_2];V)$ which follows easily from the convergence of $(u_n)_{n\in\NN}$ in $C([T_1,T_2];V)$ and by the boundedness of this sequence in $C^{\beta,\sim}([T_1,T_2];V)$. Indeed, for every $s<t\in (T_1,T_2]$

\begin{equation*}
    (s-T_1)^\beta\frac{|(u_n-u_0)(t)-(u_n-u_0)(s)|}{|t-s|^\beta}=\lim_{m\to\infty}  (s-T_1)^\beta\frac{|(u_m-u_n)(t)-(u_m-u_n)(s)|}{|t-s|^\beta}.
\end{equation*}

The elements of the sequence on the right hand side are uniformly bounded by $\|u_m-u_n\|_{\beta,\sim}$ and hence

\begin{equation*}
    \|u_{n}-u_0\|_{\beta,\sim}\le \liminf_{m\to\infty}\|u_m-u_n\|_{\beta,\sim}=:Y_n
\end{equation*}

where $(Y_n)_{n\in\NN}$  tends to zero for $n\to\infty$.
\end{proof}

In the following we wish to define the integral

\begin{equation*}
    \int_{T_1}^{T_2} Zd\omega
\end{equation*}

where $\omega$ is a H{\"o}lder continuous path.

Let us assume that $\tilde V,\,\hat V$ are separable Hilbert spaces, then for $0<\alpha<1$ and general measurable functions $Z:[T_1,T_2]\mapsto \hat V$ and $\omega:[T_1,T_2]\mapsto \tilde V$, we define their Weyl fractional derivatives by
\begin{align*}
    D_{{T_1}+}^\alpha Z[r]&=\frac{1}{\Gamma(1-\alpha)}\bigg(\frac{Z(r)}{(r-T_1)^\alpha}+\alpha\int_{T_1}^r\frac{Z(r)-Z(q)}{(r-q)^{1+\alpha}}dq\bigg)\in \hat V,\,\\
    D_{{T_2}-}^{1-\alpha} \omega_{T_2-}[r]&=\frac{(-1)^{1-\alpha}}{\Gamma(\alpha)}
    \bigg(\frac{\omega(r)-\omega(T_2-)}{(T_2-r)^{1-\alpha}}
    +(1-\alpha)\int_r^{T_2}\frac{\omega(r)-\omega(q)}{(q-r)^{2-\alpha}}dq\bigg)\in
    \tilde V,
\end{align*}
where $ \omega_{T_2-}(r)= \omega(r)- \omega(T_2-)$, being $\omega(T_2-)$ the left side limit of $\omega$ at $T_2$. \\

Suppose that $z(T_1+),\,\zeta(T_1+),\,\zeta(T_2-)$ exist, being respectively the right side limit of $z$ at $T_1$ and the right and left side limits of $\zeta$ at $T_1,\,T_2$, and that $z_{T_1+} \in I_{T_1+}^\alpha (L_p((T_1,T_2);\mathbb R)),\, \zeta_{T_2-} \in
I_{T_2-}^{\alpha} (L_{p^\prime}((T_1,T_2); \mathbb R))$ with $1/p+1/{p^\prime}\le 1$ (see the definition of these spaces in Samko {\it et al.} \cite{Samko}). Then following Z\"ahle \cite{Zah98} we define

\begin{align*}\label{eq10bi}
\begin{split}
    \int_{T_1}^{T_2} zd\zeta&=(-1)^\alpha\int_{T_1}^{T_2} D_{T_1+}^\alpha z_{T_1+}[r]D_{T_2-}^{1-\alpha}\zeta_{T_2-}[r]dr+z(T_1+)(\zeta(T_2-)-\zeta(T_1+)),
\end{split}
\end{align*}
where   $ z_{T_1+}(r)= z(r)- z(T_1+)$ and $ \zeta_{T_2-}(r)= \zeta(r)- \zeta(T_2-)$, for $r\in (T_1,T_2)$.\\

Under the previous conditions, if in addition $\alpha p<1$, then the above integral can be rewritten in a shorter way by
\begin{equation}\label{eq10bis}
    \int_{T_1}^{T_2} zd\zeta=(-1)^\alpha\int_{T_1}^{T_2} D_{T_1+}^\alpha z[r]D_{T_2-}^{1-\alpha}\zeta_{T_2-}[r]dr.
\end{equation}

In fact, under the above assumptions, the appearing fractional derivatives in the integrals are well defined taking $\tilde V=\hat V=\RR$.\\

Consider now the separable Hilbert space $L_2(V)$
of Hilbert-Schmidt operators from $V$ into $V$ with the usual norm $\|\cdot\|_{L_2(V)}$ and inner product $(\cdot,\cdot)_{L_2(V)}$. A base in this space is given by

\begin{equation}\label{base}
E_{ij}e_k=
\left\{
\begin{array}{lcl}
    0&:& j\not= k\\
    e_i &:& j= k.
    \end{array}
    \right.
\end{equation}

Let us consider now mappings $Z:[0,T]\to L_2(V)$ and $\omega:[0,T]\to V$. Suppose that $z_{ji}=(Z,E_{ji})_{L_2(V)}\in I_{T_1+}^\alpha (L_p((T_1,T_2);\mathbb R))$ and $z_{ji}(T_1+)$ exists and $\alpha p<1$. Moreover,  $\zeta_{iT_2-}=(\omega_{T_2-}(t),e_i)\in I_{T_2-}^{1-\alpha} (L_{p^\prime}((T_1,T_2); \mathbb R))$ such that $1/p+1/p^\prime\le 1$. In addition,

\begin{equation*}
   [T_1,T_2]\ni r\mapsto  \|D_{T_1+}^\alpha Z[r]\|_{L_2(V)}|D_{T_2-}^{1-\alpha} \omega_{T_2-}[r]|\in L_{1}((0,T);\RR).
\end{equation*}

We then introduce

\begin{equation}\label{eq3}
    \int_{T_1}^{T_2} Z d\omega:= (-1)^\alpha\int_{T_1}^{T_2} D_{T_1+}^\alpha Z[r]D_{T_2-}^{1-\alpha}\omega_{T_2-}[r]dr.
\end{equation}

Due to Pettis' theorem and the separability of $V$ the integrand is weakly measurable and hence measurable. In addition, we can present this integral by

\begin{equation}\label{eq36}
    \int_{T_1}^{T_2}Zd\omega=\sum_{j}\bigg(\sum_i\int_{T_1}^{T_2}
    D_{T_1+}^{\alpha}z_{ji}[r]D_{T_2-}^{1-\alpha}\zeta_{iT_2-}[r]dr \bigg) e_j,
\end{equation}
with norm given by
\begin{align*}
    \bigg|\int_{T_1}^{T_2}Zd\omega\bigg|
    &=\bigg(\sum_j\bigg|\sum_{i}\int_{T_1}^{T_2}D_{T_1+}^{\alpha}z_{ji}[r]D_{T_2-}^{1-\alpha}\zeta_{iT_2-}[r]dr\bigg|^2\bigg )^\frac12\\
    &\le \int_{T_1}^{T_2}\|D_{T_1+}^\alpha Z[r]\|_{L_2(V)}|D_{T_2-}^{1-\alpha} \omega_{T_2-}[r]| dr.
\end{align*}

 Since these one dimensional integrals under the sums are generalizations of the classical integral, i.e. if $\zeta_i$ were in $C^1$, we can interpret \eqref{eq3} to be an extension to classical separable Hilbert space valued integrals.\\

\begin{remark}\label{MN}
Later on $\omega$ will be given by a fractional Brownian motion $B^H$ with Hurst index $H>1/2$, in the way
\begin{equation*}
 B^H(t)=\sum_{i=1}^{\infty} \sqrt{q_i}e_i \beta_i^H(t),\quad t\in\mathbb{R},
\end{equation*}
where $(\beta_i^H(t))_{i\in{\mathbb N}}$ is a sequence of stochastically independent one-dimensional fBm and $\sum_{i=1}^{\infty}q_i <\infty$, see Section 4 for a detailed introduction of this process. In that situation, under the extra condition that
$$\sum_{i=1}^{\infty}\sqrt{q_i} <\infty,$$
the definition of the corresponding integral, i.e.
\begin{equation*}
\int_{T_1}^{T_2}ZdB^H
\end{equation*}
was given in \cite{MasNua03}. In the current article we do not need to require such a strong regularity condition for the noise, but the price we have to pay is that we need to consider a bit more regular integrand functions $Z$, which belong to the Hilbert-Schmidt space $L_2(V)$.

\end{remark}

In what follows, for $H>1/2$ {\footnote{This value $H$ will be in Section \ref{sRDS} the Hurst parameter of a fractional Brownian motion.}} such that in fact $1/2<\beta<\beta^\prime<H$, let $\Omega$ be the $(\theta_t)_{t\in\RR}$-invariant set of {\em paths} $\omega:\RR\to V$ which are $\beta^{\prime}$-H{\"o}lder continuous on any compact subinterval of $\RR$, being  zero at zero.
For the flow $(\theta_t)_{t\in\RR}$ on $\Omega$ of non-autonomous perturbations we consider the so-called Wiener shifts given by
\begin{align}\label{shift}
&\theta:\RR\times \Omega\to\Omega, \qquad \theta_t \omega(\cdot)=\omega(t+\cdot)-\omega(t).
\end{align}

\begin{lemma}\label{l3}
Suppose that $Z\in C^{\beta}([T_1,T_2];L_2(V))$ and $\omega\in \Omega$ such that $1-\beta^\prime<\alpha<{\beta}$.  Then

\[
\int_{T_1}^{T_2} Z d\omega\in V
\]

is well-defined in the sense of (\ref{eq3}). In addition, there exists a constant $c$ depending only on $T_2,\,\beta,\,\beta^\prime$ such that
\begin{align*}
\bigg|\int_{T_1}^{T_2} Z d\omega\bigg|&\le
     c \|Z\|_\beta |||\theta_{T_2}\omega|||_{\beta^\prime,T_1-T_2,0}(T_2-T_1)^{{\beta^\prime}}\\
     &=
     c \|Z\|_\beta |||\theta_{T_1}\omega|||_{\beta^\prime,0,T_2-T_1}(T_2-T_1)^{{\beta^\prime}}\\
     & = c \|Z\|_\beta |||\omega|||_{\beta^\prime,T_1,T_2}(T_2-T_1)^{{\beta^\prime}}.
\end{align*}
\end{lemma}
\begin{proof}
Remembering that $z_{ji}=(Z,E_{ji})_{L_2(V)}$ where $E_{ij}$ denotes the element of the basis in $L_2(V)$ given by (\ref{base}), we have
\begin{align*}
   \bigg( &\sum_{ij}| D_{T_1+}^\alpha z(\cdot)_{ji}[r]|^2\bigg)^\frac12\\
   &= \bigg(\sum_{ij}\bigg(\frac{1}{\Gamma(1-\alpha)}\bigg(\frac{z_{ji}(r)}{(r-T_1)^\alpha}
   +\alpha\int_{T_1}^r\frac{z_{ji}(r)-z_{ji}(q)}{(r-q)^{1+\alpha}}dq\bigg)\bigg)^2\bigg)^\frac12\\
   &\le \sqrt{2} c\bigg(\frac{(\sum_{ji}z_{ji}(r)^2)^\frac12}{(r-T_1)^\alpha}
   +\bigg(\sum_{ij}\bigg(\int_{T_1}^r\frac{z_{ji}(r)-z_{ji}(q)}{(r-q)^{1+\alpha}}dq\bigg)^2\bigg)^\frac12\bigg)\\
   &\le \sqrt{2}c\bigg(\frac{\|Z(r)\|_{L_2(V)}}{(r-T_1)^\alpha}+\int_{T_1}^r\frac{\|Z(r)-Z(q)\|_{L_2(V)}}{(r-q)^{1+\alpha}}dq\bigg),
\end{align*}
where we use that
 \begin{equation*}
 \bigg\| \int_{T_1}^r\frac{Z(r)-Z(q)}{(r-q)^{1+\alpha}}dq\bigg\|_{L_2(V)}\le \int_{T_1}^r\frac{\|Z(r)-Z(q)\|_{L_2(V)}}{(r-q)^{1+\alpha}}dq
\end{equation*}
and therefore, since $Z\in C^{\beta}([T_1,T_2];L_2(V))$,
\begin{equation}\label{DG}
    \|D_{T_1+}^\alpha Z[r]\|_{L_2(V)}\le c\|Z\|_\beta((r-T_1)^{-\alpha}+(r-T_1)^{\beta-\alpha}).
    \end{equation}
Similarly, for $ \omega \in \Omega$, it is straightforward to obtain that
\begin{equation}\label{Do}
|D_{T_2-}^{1-\alpha}\omega_{T_2-}[r]|\le c|||\omega|||_{\beta^\prime,T_1,T_2}(T_2-r)^{\alpha+{\beta^\prime}-1}.
\end{equation}
In fact, thanks to the definition of $(\theta_t)_{t\in\RR}$ given by (\ref{shift}), the following property follows:

\begin{align*}
    D_{{T_2}-}^{1-\alpha} \omega_{T_2-}[r]&=\frac{(-1)^{1-\alpha}}{\Gamma(\alpha)}
    \bigg(\frac{\omega(r)-\omega(T_2)}{(T_2-r)^{1-\alpha}}
    +(1-\alpha)\int_r^{T_2}\frac{\omega(r)-\omega(q)}{(q-r)^{2-\alpha}}dq\bigg)\\
    &=\frac{(-1)^{1-\alpha}}{\Gamma(\alpha)}
    \bigg(\frac{\omega(\tau+T_2)-\omega(T_2)}{(-\tau)^{1-\alpha}}
    +(1-\alpha)\int_{\tau+T_2}^{T_2}\frac{\omega(\tau+T_2)-\omega(q)}{(q-\tau-T_2)^{2-\alpha}}dq\bigg)\\
     &=\frac{(-1)^{1-\alpha}}{\Gamma(\alpha)}
    \bigg(\frac{\theta_{T_2}\omega(\tau)}{(-\tau)^{1-\alpha}}
    +(1-\alpha)\int_{\tau}^{0}\frac{\theta_{T_2}\omega(\tau)-\theta_{T_2}\omega(q)}{(q-\tau)^{2-\alpha}}dq\bigg)\\
    &=D_{{0}-}^{1-\alpha} (\theta_{T_2}\omega)_{0-}[\tau],
    \end{align*}
where $\tau=r-T_2 \in [T_1-T_2,0]$. From (\ref{Do}), since $\theta_{T_2}\omega \in C^{\beta^\prime} ([T_1-T_2,0];V)$ we also know that $|D_{{0}-}^{1-\alpha} (\theta_{T_2}\omega)_{0-}[\tau]|\le c|||\theta_{T_2}\omega|||_{\beta^\prime,T_1-T_2,0}(T_2-r)^{\alpha+{\beta^\prime}-1}$ and therefore, for any $\omega \in \Omega$ we obtain

\begin{equation}\label{Domega}
|D_{T_2-}^{1-\alpha}\omega_{T_2-}[r]|\le c|||\theta_{T_2}\omega|||_{\beta^\prime,T_1-T_2,0}(T_2-r)^{\alpha+{\beta^\prime}-1}.
\end{equation}

Thus, combining (\ref{DG}) and (\ref{Domega}), this leads to
\begin{align*}\label{eq5}
\begin{split}
    &\bigg|\int_{T_1}^{T_2} Z d\omega\bigg|\\
    \le& c \|Z\|_\beta |||\theta_{T_2}\omega|||_{\beta^\prime,T_1-T_2,0}\int_{T_1}^{T_2}((r-T_1)^{-\alpha}+(r-T_1)^{-\alpha+\beta})(T_2-r)^{\alpha+{\beta^\prime}-1}dr\\
    \le& c \|Z\|_\beta |||\theta_{T_2}\omega|||_{\beta^\prime,T_1-T_2,0}((T_2-T_1)^{{\beta^\prime}}+(T_2-T_1)^{{\beta+\beta^\prime}})\\
    \le&
     c \|Z\|_\beta |||\theta_{T_2}\omega|||_{\beta^\prime,T_1-T_2,0}(T_2-T_1)^{{\beta^\prime}}.
    \end{split}
\end{align*}

The constant $c$ appearing above depends on $\beta$ and $\beta^\prime$.

Note that the left hand side is independent of the choice of $\alpha$ contained in an appropriate interval.

\end{proof}

\begin{remark}\label{r2}
As a generalization of Z{\"a}hle \cite{Zah98} Theorem 2.5 we have the additivity of the integrals:

\begin{equation*}
    \int_{T_1}^{T_2} Z d\omega+\int_{T_2}^{T_3} Z d\omega=\int_{T_1}^{T_3} Z d\omega\quad\text{for }T_1<T_2<T_3.
\end{equation*}
\end{remark}

Furthermore, for the set $\Omega$ introduced above and the flow $\theta$ defined by (\ref{shift}), we can also establish the behavior of the stochastic integral when performing a change of variable, which is a generalization of Lemma 5 in \cite{GLS09}. Nevertheless, for the sake of completeness, we include here a short proof of it.

\begin{remark}\label{r3} For any $\tau \in \RR$ yields
\begin{equation*}
    \int_{T_1}^{T_2} Z(r) d\omega(r)=\int_{T_1-\tau}^{T_2-\tau} Z(r+\tau) d\theta_\tau \omega(r).
\end{equation*}
\end{remark}

\begin{proof} We know that
\begin{equation*}
    \int_{T_1}^{T_2} Z(r) d\omega(r)=\sum_{j}\bigg(\sum_i\int_{T_1}^{T_2}
    D_{T_1+}^{\alpha}z_{ji}[r]D_{T_2-}^{1-\alpha}\zeta_{iT_2-}[r]dr \bigg) e_j,
\end{equation*}
where $z_{ji}=(Z,E_{ji})_{L_2(V)}$ and $\zeta_{iT_2-}=(\omega_{T_2-}(t),e_i)$ have been introduced previously in the construction of the integral. Taking into account the definition of the fractional derivatives and the expression of the Wiener shift, making the change of variables $s=r-\tau$ and afterwards renaming $s$ as $r$, we have

\begin{align*}
\int_{T_1}^{T_2}
    D_{T_1+}^{\alpha}z_{ji}[r]D_{T_2-}^{1-\alpha}\zeta_{iT_2-}[r]dr =\int_{T_1-\tau}^{T_2-\tau} D_{{(T_1-\tau)}+}^{\alpha} {z_{ji}(\tau+\cdot)[r]}D_{(T_2-\tau)-}^{1-\alpha}\theta_\tau \zeta_{i(T_2-\tau)-}[r]dr.
    \end{align*}
Therefore
\begin{align*}
    \int_{T_1}^{T_2} Z(r) d\omega(r)&=\sum_{j}\bigg(\sum_i\int_{T_1-\tau}^{T_2-\tau} D_{{(T_1-\tau)}+}^{\alpha} {z_{ji}(\tau+\cdot)[r]}D_{(T_2-\tau)-}^{1-\alpha}\theta_\tau \zeta_{i(T_2-\tau)-}[r]dr \bigg) e_j\\
    &= \int_{T_1-\tau}^{T_2-\tau} Z(r+\tau) d\theta_\tau \omega(r).
\end{align*}\end{proof}

In the following we consider non-linear operators $G: V\to L_2(V)$ with appropriate regularity assumptions, which allows to establish the next collection of estimates:
\begin{lemma}\label{l2}
Let $G: V\to L_2(V)$ be a  twice continuously
Fr\'echet--differentiable operator with bounded first and second
derivatives. Let us denote, respectively, by $c_{DG},\, c_{D^2G}$
the bounds for these derivatives and set $c_G=\|G(0)\|_{L_2(V)}$.
Then, for $u_1,\,u_2,\,v_1,\,v_2\in V$, we have
\begin{align*}
    &\|G(u_1)\|_{L_2(V)}\le c_G+c_{DG}|u_1|,\\
    &\|G(u_1)-G(v_1)\|_{L_2(V)}\le c_{DG}|u_1-v_1|,\\
    &\|G(u_1)-G(v_1)-(G(u_2)-G(v_2))\|_{L_2(V)}\\
    & \quad \le c_{DG}|u_1-v_1-(u_2-v_2)|+c_{D^2G} |u_1-u_2|(|u_1-v_1|+|u_2-v_2|).
\end{align*}
\end{lemma}

The proof of this Lemma is straightforward, see for instance Maslowski and Nualart \cite{MasNua03}.

In the next lemma we estimate the fractional derivative of a very specific term which will appear later.
\begin{lemma}\label{l8}
Let $S$ be the semigroup generated by $-A$ introduced at the beginning of this subsection and assume that $G$ satisfies the assumptions of Lemma \ref{l2}. Assume that $u\in C^{\beta,\sim}([T_1,T_2];V)$ for $T_1\ge 0$.
Then for the orthonormal base $(e_i)_{i\in\NN}$ of $V$ introduced at the beginning of this section, for any $i,\,j\in\NN$,
\begin{equation*}
     (e_i,S(t-\cdot)G(u(\cdot))e_j) \in I_{T_1+}^\alpha(L_p((T_1,t);\RR)).
\end{equation*}
In addition, the mapping $r\mapsto D_{T_1+}^\alpha S(t-\cdot)G(u(\cdot))[r]$ is measurable on $[T_1,t]$ for $t\le T_2$ and satisfies the estimate
\begin{equation*}
  ||D_{T_1+}^\alpha S(t-\cdot)G(u(\cdot))[r]||_{L_2(V)}\le c (1+\|u\|_{\beta,\sim})(r-T_1)^{-\alpha}\bigg(1+\frac{(r-T_1)^\beta}{(t-r)^\beta}\bigg).
\end{equation*}
\end{lemma}

\begin{proof}

We restrict ourselves to the case $T_1=0$. We want to prove that
\begin{equation*}
    e^{-\lambda_i(t-\cdot)}(e_i,G(u(\cdot))e_j)\in I_{0+}^\alpha(L_p((0,t);\RR)),
\end{equation*}

$p>1$. Since $e^{-\lambda_i(t-\cdot)}$ is Lipschitz on $[0,t]$ it is enough to consider a Lipschitz function $g:V\to \RR$, where the Lipschitz constant is denoted by $L_g$, and ask whether $g(u(\cdot))\in I_{0+}^\alpha(L_p((0,t);\RR))$. To see this property we apply Samko {\it et al.}  \cite{Samko} Theorem 13.2. Trivially $g(u(\cdot))\in L_p((0,t);\RR)$ and $g(u(0))$ exists since $u$ is continuous on $[0,t)$. Stated as before $\beta>\alpha$ (see Lemma \ref{l3}), choosing $p>1$  such that $p\beta <1$ (and hence $p\alpha<1$) and $p(1+\alpha-\beta)<1$, and considering

\begin{equation*}
    \psi_\eps(r)=\left\{\begin{array}{ccc}
                   \int_0^{r-\eps}\frac{g(u(r))-g(u(q))}{(r-q)^{1+\alpha}}dq & : & r> \eps \\
                    \frac{g(u(r))}{\alpha}\bigg(\frac{1}{\eps^\alpha}-\frac{1}{r^\alpha}\bigg)& : & 0\le r\le \eps
                 \end{array}\right.
\end{equation*}

we want to show that the function $\psi_\eps$ converges to $\psi_0$ given by the first part of the definition of $\psi_\eps$ for $\eps=0$ in $L_p((0,t);\RR)$ for $\eps\to 0$. Take at first $\eps<r$, then

\begin{align*}
 \int_\eps^t|\psi_\eps(r)-\psi_0(r)|^pdr &\le L_g\|u\|_{\beta,\sim}\int_\eps^t\bigg(\int_{r-\eps}^r\frac{(r-q)^\beta}{(r-q)^{1+\alpha}q^\beta}dq\bigg)^pdr\\
 &\le
L_g\|u\|_{\beta,\sim}\int_\eps^t \int_{r-\eps}^r\frac{(r-q)^{p\beta}}{(r-q)^{p(1+\alpha)}q^{p\beta}}dq \eps^{p/p^\prime}dr
\end{align*}

where $p^\prime$ is the conjugate exponent of $p$. Since the interior integral can be enlarged to an integral from 0 to $r$ which is finite if $p$ is close to 1, we have that

\begin{equation*}
  \int_\eps^t|\psi_\eps(r)-\psi_0(r)|^pdr \le c   \eps^{p/p^\prime}\int_0^t r^{1-(1+\alpha)p}dr
\end{equation*}

that converges to zero for $\eps\to 0$. Furthermore, for $r\in [0,\eps]$, the convergence of $\psi_\eps $ to $ 0$ in $L_p((0,t);\RR)$ follows thanks to the boundedness of $g(u(\cdot))$ over that interval for $p$ close to 1.

To see the a priori estimate for $D_{T_1+}^\alpha S(t-\cdot)G(u(\cdot))[r]$ we refer to \eqref{l11} where a similar estimate is derived.

 \end{proof}
We also can state  that $\zeta_{iT_2-}=(\omega_{T_2-}(t),e_i)$ is in $I_{T_2-}^{1-\alpha}(L_{p^\prime}((T_1,T_2);\RR))$ for {\em any} $p^\prime>1$ because this function is $\beta^\prime$--H{\"o}lder continuous. Note that $\zeta_{iT_2-}\in L_{p^\prime}((T_1,T_2);\RR))$ and in addition
\begin{equation*}
  r\mapsto \frac{\zeta_{iT_2-}(r)}{(T_2-r)^{1-\alpha}}
\end{equation*}
is $\beta^\prime+\alpha-1$--H{\"o}lder-continuous when we augment this definition by 0 at $r=T_2$, which is based on the fact that $\zeta_{iT_2-}(T_2)=0$, which follows from $\omega_{T_2-}(T_2)=0$.
Hence this function is in $L_{p^\prime}((T_1,T_2);\RR))$ for any $p^\prime>1$.
We can also check the other conditions in  Samko \cite{Samko} Theorem 13.2 such that we get that $\zeta_{iT_2-}\in I_{T_2-}^{1-\alpha}(L_{p^\prime}((T_1,T_2);\RR))$.  Hence the integral
\begin{equation}\label{final}
\int_{T_1}^{T_2}D_{T_1+}^\alpha S(t-\cdot)G(u(\cdot))[r]D_{T_2-}^{1-\alpha}\omega_{T_2}[r]dr
\end{equation}
is well defined in the sense of \eqref{eq36}. In particular by Lemma \ref{l8}, by the fact that $(e_i,S(t)G(u(0))e_j)$ is bounded and the previous discussion, it make sense to define the integral \eqref{final} by means of its components.
\section{Non--autonomous dynamical systems of evolution equations driven by an integral with H{\"o}lder continuous integrator}\label{N-Ads}\label{s3}

We now consider the following evolution equation on $[0,T]$:

\begin{equation}\label{eq6}
    du=Au\,dt+F(u)\,dt+G(u)\,d\omega,\qquad u(0)=u_0\in V
\end{equation}

driven by a H{\"o}lder continuous path $\omega$ with H\"older exponent greater than $1/2$.

This equation is interpreted in the mild sense such that for $t\in [0,T]$ we have to solve

\begin{equation}\label{eq7}
    u(t)=S(t)u_0+\int_0^tS(t-r)F(u(r))dr+\int_0^tS(t-r)G(u(r))d\omega.
\end{equation}

The non-linear term $G$ satisfies the assumptions of Lemma \ref{l2}. The mapping $F:V\to V$ is supposed to be Lipschitz continuous, but in fact we assume that $F\equiv 0$; this is a simplification that we make for brevity, and of course that we would achieve the same existence results as we obtain below assuming that $F$ were Lipschitz. \\

The integral with respect to $d\omega$ is interpreted in the sense of the previous section.
Let $u$ in $C^{\beta,\sim}([0,T];V)$ and denote by
$\tT(\cdot,\omega,u_0)$ the operator defined on $C^{\beta,\sim}([0,T];V)$ given by the right hand side of (\ref{eq7}).

Note that $\tT$ also depends on the time interval $[0,T]$, but to make simpler the notation we do not indicate this dependence in the notation of $\tT$.

We also note that we can consider the system (\ref{eq6}) on other general time intervals $[T_1,T_2]$.
\medskip

Existence of this kind of equation has been investigated by Maslowski and Nualart \cite{MasNua03} when considering as integrators of the stochastic integrals a fractional Brownian motion with Hurst parameter in $(1/2,1)$. However, as we pointed out in the Introduction and in Remark \ref{MN}, in this article we want to present the existence theory in other function spaces, namely the space of H{\"o}lder continuous functions for appropriate exponents.

\medskip

 We need the following technical lemma.

\begin{lemma}\label{l16} Let $a>-1,\,b>-1$ and $a+b\ge -1,\,d> 0$ and $t \in[0,T]$. If for $\rho >0$ we define

\begin{equation*}
     K(\rho):=\sup_{t \in[0,T]}t^d\int_0^1e^{-\rho t(1-v)}v^{a}(1-v)^bdv,
     \end{equation*}

then we have that $\lim_{\rho\to\infty}K(\rho)=0$.
\end{lemma}
\begin{proof}
The result follows easily since

\begin{align*}
t^d \int_0^1 e^{-\rho t (1-v) }v^{a}(1-v)^b dv
\le\left\{
\begin{array}{lcr}
\rho^{-\frac{d}{2}}\int_0^1v^{a}(1-v)^{b}dv&:&\text{for }t \le\rho^{-\frac12}\\
t^d \int_0^1e^{-\rho^\frac12(1-v)}v^{a}(1-v)^{b}dv&:&\text{for }t>\rho^{-\frac12}
\end{array}\right.\\
\end{align*}
which means that

\begin{align*}
\begin{split}
K(\rho)
\le &\sup_{t\in [0,T]}\left\{
\begin{array}{lcr}
\rho^{-\frac{d}{2}}B(a+1,b+1)&:&\text{for }t\le\rho^{-\frac12}\\\\
t^d e^{-\rho^\frac12} \displaystyle{\frac{\Gamma(a+1)\Gamma(b+1)}{\Gamma(b+a+2)}} {}_1F_1(a+1,b+a+2,\rho^{\frac{1}{2}})&:&\text{for }t>\rho^{-\frac12}
\end{array}\right\}\\
\end{split}
\end{align*}

where $B(\cdot,\cdot)$ denotes the Beta function, $\Gamma(\cdot)$ the Gamma function and

$$ {}_1F_1(x,y,z):= \frac{\Gamma(y)}{\Gamma(y-x) \Gamma(x)}\int_0^1 e^{zv}v^{x-1} (1-v)^{y-x-1} \,dv,$$
is the Kummer function or hypergeometric function. The property on the convergence of $K(\rho)$ follows from the asymptotic properties of the Kummer function, see for instance Chapter 13 in \cite{abraste}. In particular, according to property 13.1.4. in  \cite{abraste} we have

$$ {}_1F_1(x,y,z)= \frac{\Gamma(y)}{\Gamma(x)}e^z z^{x-y} (1+O(|z|^{-1}))$$

which implies
\begin{align*}
K(\rho)
\le& \max(\rho^{-\frac{d}{2}}B(a+1,b+1),T^d \Gamma (b+1) \rho^{- \frac{1}{2}(b+1)} (1+O(\rho^{-\frac12}))),
\end{align*}

and therefore the convergence holds true since $b>-1$.
\end{proof}

Of course the function $K$ introduced in the last lemma also depends on the parameters $d$, $a$ and $b$, however in what follows we do not indicate this dependence except the one on $\rho$.\\

The next result will be crucial when proving the existence of solutions to (\ref{eq6}) by using the Banach fixed point theorem. Remember that we have chosen $1/2<\beta<\beta^\prime<H$ with $1-\beta^\prime<\alpha<\beta$.

\begin{lemma}\label{l6}
For any $T>0$
there exists a $c_T>0$ such that for $\omega\in \Omega$ and $u\in C^{\beta,\sim}([0,T];V)$
\begin{equation}\label{eq20}
    \|\tT(u,\omega,u_0)\|_{\beta,\rho,\sim}\le c_T|||\omega|||_{\beta^\prime,0,T}K(\rho)(1+\|u\|_{\beta,\rho,\sim})+c|u_0|.
\end{equation}
\end{lemma}

\begin{proof} By the definition of the norm and of $\tT$,
\begin{align}\label{equ31}
\begin{split}
    \|\tT & (u,\omega,u_0)\|_{\beta,\rho,\sim}
    \leq \sup_{t\in[0,T]}e^{-\rho t}\bigg|\int_0^tS(t-r)G(u(r))d\omega\bigg|\\
    &+  \sup_{0< s<t\leq T} \frac{s^\beta e^{-\rho t}}{{|t-s|^\beta}} \bigg|\int_s^tS(t-r)G(u(r))d\omega\bigg|\\
   &+  \sup_{0< s<t\leq T} \frac{s^\beta e^{-\rho t}}{{|t-s|^\beta}}\bigg|\int_0^s(S(t-r)-S(s-r))G(u(r))d\omega\bigg|\\
    &+ \sup_{t\in[0,T]}e^{-\rho t}|S(t)u_0|+\sup_{0<s<t\leq T}s^\beta e^{-\rho t}\frac{|S(t)u_0-S(s)u_0|}{|t-s|^\beta}.
\end{split}
\end{align}

Since
\begin{equation*}
    |D_{t-}^{1-\alpha}\omega[r]|\le c |||\omega|||_{\beta^\prime,0,T}(t-r)^{\alpha+\beta^\prime-1},
\end{equation*}

by using the inequalities of Lemma \ref{l2} and \eqref{eq30} we get
\begin{align}\label{l11}
\begin{split}
   & s^\beta  e^{-\rho t}\bigg|\int_s^tS(t-r)G(u(r))d\omega\bigg|\\     &\le c s^\beta e^{-\rho t}
    \int_s^t\bigg(\frac{\|S(t-r)\|_{L(V)}\|G(u(r))\|_{L_2(V)}}{(r-s)^\alpha}\\
    &+\int_s^r\frac{\|S(t-r)-S(t-q)\|_{L(V)}\|G(u(r))\|_{L_2(V)}}{(r-q)^{1+\alpha}}dq\\
    &+\int_s^r\frac{\|S(t-q)\|_{L(V)}\|G(u(r))-G(u(q))\|_{L_2(V)}}{(r-q)^{1+\alpha}}dq
    \bigg)|||\omega|||_{\beta^\prime,0,T}(t-r)^{\alpha+\beta^\prime-1}dr\\
& \leq  c T^\beta  |||\omega|||_{\beta^\prime,0,T}\bigg(
\int_s^t e^{-\rho(t-r)}\frac{(c_G+c_{DG}|u(r)|)e^{-\rho r}}{(r-s)^{\alpha}}(t-r)^{\alpha+\beta^\prime-1}dr\\
&+\int_s^t\int_s^re^{-\rho(t-r)}\frac{ e^{-\rho r}(c_G+c_{DG}|u(r)|)(r-q)^{\beta}}{(t-r)^\beta(r-q)^{1+\alpha}}dq(t-r)^{\alpha+\beta^\prime-1}dr\\
&+\int_s^t\int_s^re^{-\rho(t-r)}\frac{ e^{-\rho r}c_{DG}|u(r)-u(q)|q^\beta}{(r-q)^{1+\alpha}q^{\beta}}dq(t-r)^{\alpha+\beta^\prime-1}dr\bigg)\\
&\le c T^\beta |||\omega|||_{\beta^\prime,0,T}(t-s)^{\beta^\prime}
(1+\|u\|_{\beta,\rho,\sim})\int_s^te^{-\rho(t-r)}(r-s)^{-\alpha}(t-r)^{\alpha-1}dr
\\
&+ cT^\beta |||\omega|||_{\beta^\prime,0,T}(1+\|u\|_{\beta,\rho,\sim})\int_s^te^{-\rho(t-r)}(r-s)^{\beta-\alpha}(t-r)^{\alpha+\beta^\prime-1-\beta}dr
\\
&+ c T^\beta |||\omega|||_{\beta^\prime,0,T}(t-s)^{\beta^\prime}
\|u\|_{\beta,\rho,\sim}\int_s^te^{-\rho(t-r)}(r-s)^{-\alpha}(t-r)^{\alpha-1}dr.
\end{split}
\end{align}

Performing a change of variable, it is easy to see that
\begin{align*}
&(t-s)^{\beta^\prime}\int_s^te^{-\rho(t-r)}(r-s)^{-\alpha}(t-r)^{\alpha-1}dr\\
=&(t-s)^{\beta^\prime-\beta}(t-s)^\beta \int_0^1 e^{-\rho(t-s)(1-v)} v^{-\alpha}(1-v)^{\alpha-1}dv=(t-s)^\beta K(\rho)
\end{align*}
taking in Lemma \ref{l16} $a=-\alpha,\,b=\alpha-1$, $d=\beta^\prime-\beta$ and $t-s$ as the corresponding $t$ there. The second integral on the right hand side may be rewritten in the same way, since

\begin{align*}
   \int_s^te^{-\rho(t-r)}(r-s)^{\beta-\alpha}(t-r)^{\alpha+\beta^\prime-1-\beta}dr
   \leq (t-s)^{\beta^\prime}\int_s^te^{-\rho(t-r)}(r-s)^{-\alpha}(t-r)^{\alpha-1}dr.
\end{align*}
Therefore, coming back to (\ref{l11}) we obtain
\begin{align*}
    s^\beta & e^{-\rho t}\bigg|\int_s^tS(t-r)G(u(r))d\omega\bigg| \le c_T |||\omega|||_{\beta^\prime,0,T} (t-s)^\beta K(\rho) (1+\|u\|_{\beta,\rho,\sim}).
\end{align*}

In a similar manner than before for the first expression on the right hand side of \eqref{equ31} we obtain
\begin{equation*}
    e^{-\rho t}\bigg|\int_0^tS(t-r)G(u(r))d\omega\bigg|\ \leq c_T|||\omega|||_{\beta^\prime,0,T}K(\rho)(1+\|u\|_{\beta,\rho,\sim}).
\end{equation*}

For the third term on the right hand side of (\ref{equ31}) we should follow similar steps than before when obtaining (\ref{l11}). Now we need to replace the
estimates for $\|S(t-r)\|_{L(V)}$ and $\|S(t-r)-S(t-q)\|_{L(V)}$ by estimates for $\|S(t-r)-S(s-r)\|_{L(V)}$ and $\|S(t-r)-S(t-q)-(S(s-r)-S(s-q))\|_{L(V)}$ respectively, for which we use \eqref{eq2} and \eqref{eq30} for appropriate parameters. Then it is not hard to see that for $\alpha^\prime+\beta<\alpha+\beta^\prime,\,0<\alpha<\alpha^\prime<1$:

\begin{align*}
\begin{split}
s^\beta e^{-\rho t} & \bigg|\int_0^s(S(t-r)-S(s-r))G(u(r))d\omega\bigg|
\\
   &\le c(t-s)^\beta|||\omega|||_{\beta^\prime,0,T}T^\beta\bigg(
\int_0^s e^{-\rho(t-r)}\frac{(c_G+c_{DG}|u(r)|)e^{-\rho r}}{(s-r)^{\alpha+\beta}}(s-r)^{\alpha+\beta^\prime-1}dr\\
&+\int_0^s\int_0^re^{-\rho(t-r)}\frac{ e^{-\rho r}(c_G+c_{DG}|u(r)|)(r-q)^{\alpha^\prime}}{(s-r)^{\alpha^\prime+\beta}(r-q)^{1+\alpha}}dq(s-r)^{\alpha+\beta^\prime-1}dr\\
&+\int_0^s\int_0^re^{-\rho(t-r)}\frac{ e^{-\rho r}c_{DG}|u(r)-u(q)|q^\beta}{(s-r)^\beta (r-q)^{1+\alpha}q^\beta}dq(s-r)^{\alpha+\beta^\prime-1}dr\bigg)\\
&\le c(t-s)^\beta T^\beta|||\omega|||_{\beta^\prime,0,T}(1+\|u\|_{\beta,\rho,\sim})\int_0^se^{-\rho(t-r)}(s-r)^{\beta^\prime-1-\beta}dr
\\
&+ c(t-s)^\beta T^\beta|||\omega|||_{\beta^\prime,0,T}(1+\|u\|_{\beta,\rho,\sim})\int_0^se^{-\rho(t-r)}r^{\alpha^\prime-\alpha}(s-r)^{\alpha+\beta^\prime-1-\alpha^\prime-\beta}dr
\\
&+ c(t-s)^\beta T^\beta|||\omega|||_{\beta^\prime,0,T}
\|u\|_{\beta,\rho,\sim}\int_0^se^{-\rho(t-r)}r^{-\alpha}(s-r)^{\alpha-\beta+\beta^\prime-1}dr.
\end{split}
\end{align*}
The third integral on the right hand side of the last inequality can be estimated by

\begin{equation*}
    s^{\beta^\prime-\beta}\int_0^1e^{-\rho s(1-v)}v^{-\alpha}(1-v)^{\alpha-1}dv
\end{equation*}
and in a similar manner the other integrals.
All the previous estimates imply that

\begin{equation*}
    \bigg\|\int_s^tS(t-r)G(u(r))d\omega\bigg\|_{\beta,\rho,\sim}\le c_T|||\omega|||_{\beta^\prime,0,T}K(\rho)(1+\|u\|_{\beta,\rho,\sim}).
\end{equation*}

Now we estimate the terms concerning the initial data. Notice that because of the properties of the semigroup (\ref{eq1}) and (\ref{eq2}) we have
\begin{align*}
{|S(t)u_0-S(s)u_0|}\leq {|(S(t-s)-{\rm{id}}) S(s)u_0|}\leq {s^{-\beta} (t-s)^{\beta}|u_0|},
\end{align*}
and therefore
\begin{align*}\begin{split}
 \sup_{t\in[0,T]}e^{-\rho t}|S(t)u_0|+\sup_{0< s<t\leq T}s^\beta e^{-\rho t}\frac{|S(t)u_0-S(s)u_0|}{|t-s|^\beta}\leq c|u_0|.
\end{split}
\end{align*}
Finally we note that $S(\cdot)u_0\in V$ is continuous.
\end{proof}

\begin{remark}\label{r1}
(i) The reason to study the problem of existence and uniqueness in the space $C^{\beta,\sim}([0,T];V)$ is coming from the fact that $S$ is not $\beta$-H{\"o}lder-continuous at zero in $V$. However, assuming that $u_0\in V_\beta$ then we could study the problem of existence and uniqueness in the Banach space $C^{\beta}([0,T];V)$, since then
\begin{align*}
{|S(t)u_0-S(s)u_0|}\leq {|(S(t-s)-{\rm{id}}) S(s)u_0|}\leq { (t-s)^{\beta}|u_0|_{V_\beta}}.
\end{align*}
Indeed, all appearing integrals of (\ref{eq7}) can be estimated
in this space. We omit the proof but, in fact, the factor $s^\beta$ is never used for the estimate of the H{\"o}lder-norm of the integrals.
\medskip

(ii) Following the same steps than in the proof of the above lemma, for every $T>0$ and $t\le T$, there exists a $c_T>0$ such that for $u\in C^{\beta,\sim}([0,T];V)$
we obtain
\begin{equation*}
   \bigg| \int_0^t S(t-r)G(u(r))d\omega\bigg|_{V_\beta}\le c_T |||\omega|||_{\beta^\prime,0,T}(1+\|u\|_{\beta,\sim}).
\end{equation*}

Moreover, by using (\ref{eq1}) and (\ref{eq30}), when replacing $u\in C^{\beta,\sim}([0,T];V)$ by $u\in C^{\beta}([0,T];V)$, we obtain that
\begin{equation*}
   \bigg| \int_0^t S(t-r)G(u(r))d\omega\bigg|_{V_\beta} \le c_T |||\omega|||_{\beta^\prime,0,T}(1+\|u\|_{\beta}),
\end{equation*}
where the constant $c_T$ is independent of $u$ and $\omega$.
Hence we obtain the following estimate when considering the $V$-norm:
\begin{equation*}
   \bigg| \int_0^t S(t-r)G(u(r))d\omega\bigg| \le c_T |||\omega|||_{\beta^\prime,0,T}(1+\|u\|_{\beta}),
\end{equation*}
where this constant $c_T$ is also independent of $u$ and $\omega$.

(iii) Assuming that $u_0\in V$ the corresponding solution $u$ of (\ref{eq7}) satisfies $u(t)\in V_\beta$ for every $t>0$.  The proof of this assertion follows immediately, since for $t>0$,

\begin{align*}
|u(t)|_{V_\beta}&\leq |S(t)u_0|_{V_\beta}+\bigg| \int_0^t S(t-r)G(u(r))\bigg|_{V_\beta} \\
&\leq ct^{-\beta}|u_0|+c_T |||\omega|||_{\beta^\prime,0,T}(1+\|u\|_{\beta,\sim})<\infty.
\end{align*}

\end{remark}

Now we study the norm of $\tT(u_1,\omega,u_{01})-\tT(u_2,\omega,u_{02})$ in $C^{\beta,\sim}([0,T];V)$. In order to do that, we apply the same techniques than above with the difference that now it is necessary to use the last inequality for $G$ given in Lemma \ref{l2}.

\begin{lemma}\label{l5}
For every $T>0$ there exists a $c_T>0$ such that for every $\rho>0$, $\omega\in \Omega$ and $u_1, u_2 \in C^{\beta,\sim}([0,T];V)$ with $u_1(0)=u_{01},\,u_2(0)=u_{02}$, with $u_{01},\,u_{02}\in V$,

\begin{align}\label{eq12}
\begin{split}
&\|\tT(u_1,\omega,u_{01})-\tT(u_2,\omega,u_{02})\|_{\beta,\rho,\sim} \\\le &  c_T|||\omega|||_{\beta^\prime,0,T}(1+ \|u_1\|_{\beta,\sim}+\|u_2\|_{\beta,\sim})K(\rho)\|u_1-u_2\|_{\beta,\rho,\sim}+c|u_{01}-u_{02}|.
\end{split}
\end{align}
\end{lemma}
\begin{proof}
As we have already mentioned, the proof follows by using the same techniques as in the proof of Lemma \ref{l6} and taking into account Lemma \ref{l2}. We only show how to manage the following term:
\begin{align*}
   & s^\beta e^{-\rho t}\bigg|\int_s^tS(t-r)(G(u_1(r))-G(u_2(r)))d\omega\bigg|\\
    &\le c_T|||\omega|||_{\beta^\prime,0,T}\bigg(
\int_s^t e^{-\rho(t-r)}(t-r)^{\alpha+\beta^\prime-1}\bigg(\frac{c_{DG}e^{-\rho r}|u_1(r)-u_2(r)|}{(r-s)^{\alpha}}\\
&+\int_s^r\frac{c_{DG}e^{-\rho r}|u_1(r)-u_2(r)|  (r-q)^\beta}{(r-q)^{1+\alpha}}dq\\
&+\int_s^r\frac{c_{DG} e^{-\rho r}|u_1(r)-u_1(q)-(u_2(r)-u_2(q))|q^\beta}{(r-q)^{1+\alpha}q^\beta}dq\\
&+\int_s^r\frac{c_{D^2G}e^{-\rho r}|u_1(r)-u_2(r)|(|u_1(r)-u_1(q)|+|u_2(r)-u_2(q)|)q^\beta}{(r-q)^{1+\alpha}q^\beta}dq\bigg)dr\\
&\leq c_T|||\omega|||_{\beta^\prime,0,T}  (t-s)^{\beta^\prime}
(1+\|u_1\|_{\beta,\sim}+\|u_2\|_{\beta,\sim}) \|u_1-u_2\|_{\beta,\rho,\sim}
\\&\qquad \qquad \times \int_s^te^{-\rho(t-r)}(r-s)^{-\alpha}(t-r)^{\alpha-1}dr\\
&\le c_T|||\omega|||_{\beta^\prime,0,T}(t-s)^{\beta}
(1+\|u_1\|_{\beta,\sim}+\|u_2\|_{\beta,\sim})\|u_1-u_2\|_{\beta,\rho,\sim} K(\rho),
\end{align*}
where $K(\rho)$ has been defined in the proof of lemma \ref{l6}.
\end{proof}

We are now in the position to prove the main theorem of this section.

\begin{theorem}\label{t1}
Let $u_0\in V$ and assume that $G$ satisfies the assumptions of Lemma \ref{l2}. Then for every $T>0$ the equation (\ref{eq7}) has a unique solution $u$ in $C^{\beta,\sim}([0,T];V)$.
\end{theorem}

\begin{proof}
From Lemma \ref{l6}, taking $\rho$ large enough such that $c_T|||\omega|||_{\beta^\prime, 0,T}K(\rho)<\frac12$, we obtain that $\mathcal T$ maps the
ball

\begin{equation*}
    B:=\{u\in C^{\beta,\sim}([0,T];V):\|u\|_{\beta,\rho,\sim}\le R \},\quad \mbox{ with }\, R:=1+2c|u_0|
\end{equation*}

into itself. Furthermore, by the equivalence of the norms $\|\cdot\|_{\beta,\rho,\sim}$ and $\|\cdot\|_{\beta,\sim}$, there exists a constant $R_1$ such that

\begin{equation*}
    \sup_{u\in B}\|u\|_{\beta,\sim}\le e^{\rho T} \sup_{u\in B}\|u\|_{\beta,\rho, \sim}\leq Re^{\rho T}=: R_1.
\end{equation*}

Therefore, for $u_{01}=u_{02}=u_0$, it is possible to find a $\rho_1>\rho$ such that, if $u_1,u_2\in B$ then
\begin{align*}
\|\tT(u_1,\omega,u_{0})-\tT(u_2,\omega,u_{0})\|_{\beta,\rho_1,\sim}& \le  c_T|||\omega|||_{\beta^\prime,0,T}(1+ 2R_1) K(\rho_1)\|u_1-u_2\|_{\beta,\rho_1,\sim}\\
\le \frac12 \|u_1-u_2\|_{\beta,\rho_1,\sim}
\end{align*}

which follows by (\ref{eq12}). This means that $\tT(\cdot,\omega,u_0)$ is a contraction on $B$ and its fixed point then solves (\ref{eq7}).

To see that a solution is unique {\em in general} take two solutions $u_1,\,u_2$ such that $\|u_1\|_{\beta,\sim},\,\|u_2\|_{\beta,\sim}
\le R$. Then

\begin{align*}
    \|u_1-u_2\|_{\beta,\rho,\sim} & =\|\tT(u_1,\omega,u_0)-\tT(u_2,\omega,u_0)\|_{\beta,\rho,\sim}\\
    &\le
    c_T|||\omega|||_{\beta^\prime,0,T}(1+ 2 R) K(\rho)\|u_1-u_2\|_{\beta,\rho,\sim}<\frac12\|u_1-u_2\|_{\beta,\rho,\sim}
\end{align*}

for sufficiently large $\rho$. But the above inequality is only possible provided that $u_1=u_2$.
\end{proof}

We conclude this section proving that $u$ defines a non--autonomous dynamical system.

\begin{theorem}\label{t2}
The solution of (\ref{eq7}) generates a non--autonomous dynamical system $\phi:\mathbb{R}^+\times \Omega\times V\to V$ given by
\begin{equation*}
\phi(t,\omega,u_0)=S(t)u_{0}+\int_0^t
S(t-s)G(u(s))d\omega. \end{equation*}
Moreover,  $u_0\mapsto\phi(t,\omega,u_0)$ is continuous on $V$ for $t\ge 0$ and $\omega\in \Omega$, and for $t>0,\,\omega\in \Omega$ the mapping $u_0\mapsto\phi(t,\omega,u_0)$ is compact.
\end{theorem}

\begin{proof}

In order to prove that $\phi$ is a  cocycle we will make use of Remark \ref{r3}, which establishes how the integral with H\"older continuous integrator behaves when making a change of variable: for $t,\,\tau\in \mathbb{R}^+$, $\omega\in \Omega$ and
$u_0\in V$,
\begin{equation*}
\int_{\tau}^{t+\tau}S(t+\tau-s)G(u(s))d\omega(s)=\int_{0}^{t}S(t-r)G(u(r+\tau))d\theta_{\tau}
\omega(r).
\end{equation*}
Then, for $t,\,\tau\in \mathbb{R}^+$ and $\omega\in \Omega$,
\begin{align*}
\varphi  & (t+\tau, \omega,u_0) =
S(t+\tau)u_{0}+\int_{0}^{t+\tau}S(t+\tau-s)
G(u(s))d\omega(s) \\
&=S(t) \left ( S(\tau)u_{0}+\int_{0}^{\tau}S(\tau-s)G(u(s))d\omega(s) \right) +\int_{0}^{t}S(t-r)G(u(r+\tau))d\theta_{\tau}
\omega(r).
\end{align*}
Therefore, setting $y(\cdot)=u(\cdot+\tau)$ on $[0,t]$,
\begin{align*}
\varphi & (t+\tau, \omega,u_0)=S(t)y(0)+\int_{0}^{t}S(t-r)G(y(r))d\theta_{\tau} \omega(r)\\
&=\varphi(t,\theta_\tau\omega,\cdot)\circ\varphi(\tau,\omega,u_0).
\end{align*}
It is trivial that $\varphi(0,\omega,u_0)=u_0$ and, by the parameter version of the fixed
point theorem, for fixed $(t,\omega)$, the fixed point depends
continuously on $u_0\in V$.

To see the compactness of $\phi$ for $t>0$ we consider for some $\eta>0$ the set $B(0,\eta)\subset V$. With respect to the proof of Theorem \ref{t1} we consider a $\rho$ such that $c_T|||\omega|||_{\beta^\prime,0,T}K(\rho)<1/2$. Then we know that for any
solution $u$ of \eqref{eq6} and initial condition in the ball $B(0,\eta)$ we have that $\|u\|_{\beta,\rho,\sim}\le 1+2c\eta$,
hence  $\|u\|_{\beta,\sim}\le e^{\rho T}(1+2c \eta)$. Then the compactness follows therefore by Remark \ref{r1} (iii).
\end{proof}

\section{Random dynamical system for SPDEs driven by an fBm with $H>1/2$.}\label{sRDS}

In this section we study the non-autonomous dynamical system under measurability assumptions, in order to get a random dynamical system (see Subsection \ref{s2} for the general definition). In particular, we need to introduce a metric dynamical system, and for that, as driving process for our equation, we choose  a $V$-valued fractional Brownian motion, which definition is given below. Notice that, as pointed out in the Introduction section, in order to get the cocycle property in this random setting it is crucial that the stochastic integrals (i.e., the integrals with fractional Brownian motion as  integrators), are defined in a pathwise way. This is a qualitative difference to the definition of the classical stochastic integral where the integrand is a white noise. \\


Given $H\in (0,1)$, a continuous centered Gau{\ss}ian process
$\beta^H(t)$, $t\in\mathbb{R}$, with the covariance function
\begin{equation*}
    \mathbb{E}\beta^H(t)\beta^H(s)=\frac{1}{2}(|t|^{2H}+|s|^{2H}-|t-s|^{2H}),\qquad t,\,s\in\mathbb{R}
\end{equation*}
on an appropriate probability space $(\Omega,\fF,\PP)$ is called a {\it two--sided one-dimensional fractional Brownian
motion}, and $H$ is the {\it Hurst parameter}.

\vskip0.1in

Assume that $Q$ is a bounded and symmetric linear operator on $V$ which is of trace class, i.e., for the complete orthonormal basis $(e_i)_{i\in {\mathbb N}}$ in $V$ there exists a sequence of nonnegative numbers $(q_i)_{i\in {\mathbb N}}$ such that $tr Q:=\sum_{i=1}^{\infty}q_i <\infty$. Then a continuous {\it $V$-valued fractional Brownian motion $B^H$} with  covariance operator $Q$ and Hurst parameter $H$ is defined by

\begin{equation*}
   B^H(t)=\sum_{i=1}^{\infty} \sqrt{q_i}e_i \beta_i^H(t),\quad t\in\mathbb{R},
\end{equation*}
where $(\beta_i^H(t))_{i\in{\mathbb N}}$ is a sequence of stochastically independent one-dimensional fBm.  Notice that the above series is convergent in $L^2(\Omega, \mathcal F, \mathbb P)$ since $\sum_{i=1}^{\infty}q_i <\infty$ and $\mathbb{E}(\beta^H_i(t))^2=|t|^{2H}$ for $t\in \mathbb R$.

When $H=1/2$, $B^H(t)$ is the standard Brownian motion. \\

Using the definition of $B^H(t)$, Kolmogorov's theorem ensures that $B^H$ has a continuous version. Thus we can consider the canonical interpretation of an fBm:  let
$C_0(\mathbb{R},V)$,  the space of continuous functions on $\mathbb{R}$ with values in $V$ such that are zero at zero, equipped with the compact open topology. Let $\mathcal{F}=\bB(C_0(\RR,V))$ be the associated Borel-$\sigma$-algebra and ${\mathbb P}$ the distribution of the fBm $B^H$, and $(\theta_t)_{t\in \mathbb R}$ be the flow of Wiener shifts given by (\ref{shift}).

With this choice the first part of the random dynamical system definition is achieved:

\begin{lemma}\label{l4}
$(C_0(\RR,V),\bB(C_0(\RR,V)),\PP,\theta)$ is an ergodic metric dynamical system.
\end{lemma}

The proof of this lemma can be found in \cite{MasSchm04} and  in \cite{GS11}. Indeed the previous result holds true no matter the value of the Hurst parameter in $(0,1)$.
\\

This (canonical) process has a version, denoted by $\omega$, which is ${\beta^{\prime}}$-H{\"o}lder continuous on any interval $[-k,k]$  for $\beta^{\prime}<H$, see Kunita \cite{Kunita90}, Theorem 1.4.1.

As we already did in Section \ref{ss2.2}, for $1/2<\beta<\beta^\prime<H$ such that $1-\beta^\prime < \beta < \alpha$, we denote by $\Omega \subset C_0(\RR,V)$ the set of functions which are $\beta^{\prime}$-H{\"o}lder continuous on any interval $[-k,k],\,k\in\NN$, and are zero at zero.

\begin{lemma}
We have $\Omega \in\bB(C_0(\RR,V))$  and $\PP(\Omega)=1$. In addition, $\Omega$ is $(\theta_t)_{t\in\RR}$-invariant.
\end{lemma}
\begin{proof}
First we note that
\begin{equation*}
    C_0(\RR,V)\ni\omega\mapsto |||\omega|||_{\beta^{\prime},-k,k}\in\bar \RR^+=\RR^+\cup\{+\infty\}
\end{equation*}
is measurable. Indeed we have
\begin{equation*}
    |||\omega|||_{{\beta^{\prime}},-k,k}=\sup_{-k\le s<t\le k,s,t\in\QQ}\frac{|\omega(t)-\omega(s)|}{|t-s|^{\beta^{\prime}}}.
\end{equation*}
Then
\begin{equation*}
    \Omega=\bigcap_{k\in\NN}\{\omega\in C_0(\RR,V):|||\omega|||_{\beta^{\prime},-k,k}<\infty\}\in\bB(C_0(\RR,V)).
\end{equation*}
The $(\theta_t)_{t\in\RR}$-invariance of $\Omega$ is straightforward.
\end{proof}

In what follows we consider the ergodic metric dynamical system introduced above restricted to the set $\Omega$: let $\tilde{\mathcal F}$ be the trace-$\sigma$-algebra of $\mathcal F$ with respect to $\Omega$, let $\tilde \PP$ the restriction of $\PP$ to this $\sigma$-algebra, and $\theta$ represents the restriction of the Wiener shifts to $\Omega\times \RR$. Then $(\Omega, \tilde{\mathcal F}, \tilde{\PP}, \theta)$ forms a metric dynamical system such that for every $\tilde A \in \tilde{\mathcal F}$ and every $A\in \mathcal F$ with $\tilde A=A\bigcap \Omega$ we have that $\tilde {\PP} (\tilde A)=\PP (A)$ independent of the representation by $A$. In addition, the ergodicity of $(C_0(\RR,V),\bB(C_0(\RR,V)),\PP,\theta)$ is transferred to $(\Omega, \tilde{\mathcal F}, \tilde {\PP}, \theta)$, see \cite{CGSV10} for details.\\

Once we have the adequate ergodic metric dynamical system given by $(\Omega, \tilde{\mathcal F}, \tilde {\PP}, \theta)$, we can prove the main result of this section:

\begin{theorem}\label{t20}
Assume that the driving process $\omega$ of (\ref{eq7}) is a fractional Brownian motion with Hurst index greater than $1/2$. Under the conditions of Theorem \ref{t1}, for every $u_0\in V$ there exists a unique mild solution $u\in C^{\beta,\sim}([0,T];V)$, which generates a random dynamical system $\phi:\mathbb{R}^+\times \Omega\times V\to V$ defined by
\begin{equation*}
\phi(t,\omega,u_0)=S(t)u_{0}+\int_0^t
S(t-s)G(u(s))d\omega. \end{equation*}
\end{theorem}

\begin{proof}
Having identified $\omega$ as an fBm with Hurst index greater than $H>1/2$, the existence of a unique pathwise mild solution to (\ref{eq7}) follows by the $\beta^{\prime}$-H\"older regularity of $\omega$ and Theorem \ref{t1}.

Furthermore, the pathwise definition of the integral just gave us the non--autonomous dynamical system $\varphi$ studied in Section \ref{s3}. Therefore, all we have to do is to establish the proper measurability conditions for the mapping $\varphi$, in that case the $\bB(\RR^+)\otimes \tilde {\fF} \otimes \bB(V),\bB(V)$ measurability, according to the definition introduced in Subsection \ref{s2}. It suffices to observe that, when starting the iteration procedure of the Banach fixed point theorem with a measurable initial function $u_0$, then $\phi(t,\omega,u_0)$ is a pointwise limit of measurable mappings. Moreover, the parameter version of the fixed point theorem for fixed $(t,\omega)$ ensures that the fixed point depends continuously on $u_0\in V$. These last two considerations together with Lemma
III.14 in \cite{CasVal77} allows to claim that $\phi(t,\omega,u_0)$ is measurable with respect to its three variables.
\end{proof}

\begin{remark}
Let us briefly comment what we could do in case of having a more regular initial condition, assuming $u_0\in V_\beta$. In that case we could use as phase space $C^{\beta}([0,T];V)$, since in this situation we would not need to consider the modification of this space given by $C^{\beta,\sim}([0,T];V)$. In addition, due to the regularizing effect of the  equation, the unique pathwise mild solution $u\in C^{\beta}([0,T];V)$ to (\ref{eq7}) would satisfy that $u(t)\in V_\beta$ for $t\in (0,T]$, see Remark \ref{r1}. Thus, when having such a more regular initial datum, we could prove the existence of a random dynamical system $\phi:\mathbb{R}^+\times \Omega\times V_\beta \to V_\beta$ defined by the corresponding solution to the stochastic evolution equation.
\end{remark}

In a forthcoming paper, see \cite{ChGGSch13}, we will deal with the study of the asymptotic behavior of the pathwise mild solution to (\ref{eq7}) by analyzing the existence and uniqueness of the pullback attractor (or random attractor when $\omega$ is a fBm with Hurst index greater than $1/2$) associated to the non-autonomous dynamical system $\varphi$ (random dynamical system, respectively).

\end{document}